\documentclass{amsart}
\usepackage{amsmath,amsthm}
\usepackage{amsfonts,amssymb}

\usepackage{enumerate}

\newtheorem{thm}{Theorem}[section]

\newtheorem{lem}[thm]{Lemma}
\newtheorem{prop}[thm]{Proposition}

\newtheorem{theorem}{Theorem}[section]
\newtheorem{corollary}[theorem]{Corollary}
\newtheorem{lemma}[theorem]{Lemma}
\newtheorem{proposition}[theorem]{Proposition}

\newtheorem{remark}[theorem]{Remark}

\theoremstyle{remark}
\newtheorem{rem}{Remark}[section]

\newcommand{\mR}{\mathbb{R}}
\newcommand{\mC}{\mathbb{C}}
\newcommand{\mN}{\mathbb{N}}
\newcommand{\mE}{\mathbb{E}}

\newcommand{\cM}{\mathcal{M}}
\newcommand{\cH}{\mathcal{H}}

\newcommand{\cF}{\mathcal{F}}
\newcommand{\cP}{\mathcal{P}}
\newcommand{\cC}{\mathcal{C}}

\newcommand{\cS}{\mathcal{S}}

\newcommand{\ux}{\underline{x}}

\newcommand{\uy}{\underline{y}}

\newcommand{\upx}{\partial_{\underline{x}}}
\newcommand{\upy}{\partial_{\underline{y}}}

 \def\a{{\alpha}} 
 \def\b{{\beta}}
 \def\g{{\gamma}}

 \def\l{{\lambda}}
 \def\d{{\delta}}
 
 \def\s{{\sigma}}
 \def\la{{\langle}}
 \def\ra{{\rangle}}

 \def\CF{{\mathcal F}}

 \def\NN{{\mathbb N}}

 \def\RR{{\mathbb R}}

        \def\tan{\operatorname{tan}}

\begin{document}
 
\title
{The fractional Clifford-Fourier transform}
\author{Hendrik De Bie}
\email{Hendrik.DeBie@UGent.be}
\author{Nele De Schepper}
\email{Nele.DeSchepper@UGent.be}
\address{Department of Mathematical Analysis\\Faculty of Engineering and Architecture\\Ghent University\\Galglaan 2, 9000 Gent\\ Belgium.}

\date{\today}
\keywords{Clifford-Fourier transform, fractional transform, integral kernel, Clifford analysis}
\subjclass{30G35, 42B10} 

\begin{abstract}
In this paper, a fractional version of the Clifford-Fourier transform is introduced, depending on two numerical parameters. A series expansion for the kernel of the resulting integral transform is derived. In the case of even dimension, also an explicit expression for the kernel in terms of Bessel functions is obtained. Finally, the analytic properties of this new integral transform are studied in detail.
\end{abstract}

\maketitle


\section{Introduction}
\setcounter{equation}{0}

The fractional Fourier transform is a generalization of the classical Fourier transform (FT). It is usually defined using the operator expression 
\[
\cF_{\alpha} =    e^{ \frac{i \alpha m}{2}}e^{ \frac{i \alpha}{2}(\Delta - |\ux|^{2} )}, \quad \a \in [-\pi,\pi]
\]
with $\Delta$ the Laplace operator in $\mR^{m}$. As integral transform, it can be written as
\begin{align*}
\cF_{\alpha} [f](\uy)&= \left(\pi (1- e^{-2i \alpha})\right)^{-m/2} \int_{\mR^m}  e^{-i \la \ux,\uy \ra / \sin \alpha} e^{\frac{i}{2}( \cot \alpha) (|\ux|^2 + |\uy|^2)}f(\ux) dx
\end{align*}
when $\a \not = \pm \pi$, $\a \not = 0$.
The precise origin of the fractional FT is not entirely clear. In the applied literature, Namias (\cite{MR573153}) is usually credited with its invention. However, Mustard (\cite{MR1668139}) attributes the fractional FT to Condon and Bargmann (see \cite{Con, Barg}).
In pure mathematics, the main ideas leading to the discovery of the fractional FT seem to have been around implicitly since the discovery of the so-called Mehler formula (\cite{M}), connecting the kernel of the fractional Fourier transform with a series expansion in terms of Hermite functions. In this context, one uses the term Hermite semigroup to denote the fractional FT, see e.g. \cite{MR0974332, MR0983366}.
Even in the completely different mathematical field of $C^{*}$-algebras, a special case of the fractional FT has been introduced independently, see \cite{MR2055227}.

For a detailed overview of the theory and applications of the fractional FT we refer the reader to \cite{OZA}. A perspective from the point of view of the Hermite semigroup can be found in e.g. \cite{Orsted2, ST}.

In the present paper, we introduce a fractional version of the Clifford-Fourier transform (CFT). The CFT (see \cite{MR2190678, MR2283868, AIEP, DBXu}) is a generalization of the Fourier transform in the framework of Clifford analysis. It is defined by the following exponential operator
\[
\cF_{\pm}:=e^{ \frac{i \pi m}{4}} e^{\frac{i \pi}{4}(\Delta - |\ux|^{2} \mp 2\Gamma)} = e^{ \frac{i \pi m}{4}} e^{\mp \frac{i \pi}{2}\Gamma  }e^{\frac{i \pi}{4}(\Delta - |\ux|^{2})}
\]
with $\Gamma_{\ux} := - \sum_{j<k} e_{j}e_{k} (x_{j} \partial_{x_{k}} - x_{k}\partial_{x_{j}})$. The integral kernel of this transform is not so easy to obtain (see \cite{DBXu}) and the result is a complicated formula (see the subsequent Theorem \ref{EvenExplicitCFT}).

Here, we generalize the CFT in the following way
\[
\cF_{\alpha, \b} =   e^{ \frac{i \alpha m}{2}} e^{i\b \Gamma}e^{ \frac{i \alpha}{2}(\Delta - |\ux|^{2})},
\]
where we have introduced two numerical parameters $\a$ and $\b$. These can of course be chosen equal to each other (as was done in \cite{DBDSC}, where only the two dimensional case was treated). However, choosing them independently gives more insight in the different roles they play.

This new fractional transform is interesting for several reasons. First, it leads to a 2-parameter multi-variate Mehler type formula. Second, it elucidates the role played by the two defining factors in the operator exponential definition of the CFT. Third, it paves the way for a functional calculus approach to generalized Fourier transforms in Clifford analysis, which will be reported in \cite{DBC}.

The paper is organized as follows. In section  \ref{prelim} we repeat some basic knowledge on Clifford analysis and give the explicit kernel of the CFT as determined in \cite{DBXu}. In section \ref{FracVersion} we define a fractional version of the CFT and obtain a series expansion of its kernel. Subsequently, in section \ref{ExplicitCF} we compute an explicit expression for the kernel in even dimension. The case of dimension two is treated separately. Next in section \ref{secFurtherprops} we show that the kernel satisfies a system of PDEs. In section \ref{PropertiesSection} we prove that the fractional CFT is a continuous operator on the Schwartz space when the dimension is even. Finally, in section \ref{EigenvaluesSection} we obtain the eigenvalues of the fractional CFT, which are used to prove an inversion theorem.

\section{Preliminaries}\label{prelim}
\setcounter{equation}{0}

The Clifford algebra $\cC l_{0,m}$ over $\mR^{m}$ is the algebra generated by $e_{i}$, $i= 1, \ldots, m$, under the relations
\begin{align} \label{eq:eij}
\begin{split}
&e_{i} e_{j} + e_{j} e_{i} = 0, \qquad i \neq j,\\
& e_{i}^{2} = -1.
\end{split}
\end{align}
This algebra has dimension $2^{m}$ as a vector space over $\mR$. It can be decomposed as $\cC l_{0,m} = \oplus_{k=0}^{m} \cC l_{0,m}^{k}$
with $\cC l_{0,m}^{k}$ the space of $k$-vectors defined by
\[
\cC l_{0,m}^{k} := \mbox{span} \{ e_{i_{1}} \ldots e_{i_{k}}, i_{1} < \ldots < i_{k} \}.
\]
In the sequel, we will always consider functions $f$ taking values in $\cC l_{0,m}$, unless explicitly mentioned. Such functions can be decomposed as
\begin{equation} \label{clifford_func}
f = f_{0} + \sum_{i=1}^{m} e_{i}f_{i}+ \sum_{i< j} e_{i} e_{j} f_{ij} + \ldots + e_{1} \ldots e_{m} f_{1 \ldots m} 
\end{equation}
with $f_{0}, f_{i}, f_{ij}, \ldots, f_{1 \ldots m}$ all real-valued functions on $\mR^{m}$.

The Dirac operator is given by $\upx = \sum_{j=1}^{m} \partial_{x_{j}} e_{j}$ and the vector variable by $\ux = \sum_{j=1}^{m} x_{j} e_{j}$.
The square of the Dirac operator equals, up to a minus sign, the Laplace operator in $\mR^{m}$: $\upx^{2} = - \Delta$.

We further introduce the so-called Gamma operator (see e.g. \cite{MR1169463})
\[
\Gamma_{\ux} := - \sum_{j<k} e_{j}e_{k} (x_{j} \partial_{x_{k}} - x_{k}\partial_{x_{j}}). 
\]
Note that $\Gamma_{\ux}$ commutes with radial functions, i.e. $[\Gamma_{\ux}, f(|\ux|)] =0$.

Denote by $\cP$ the space of polynomials taking values in $\cC l_{0,m}$, i.e. 
$$
   \cP : = \mR[x_{1}, \ldots, x_{m}] \otimes \cC l_{0,m}.
$$ 
The space of homogeneous polynomials of degree $k$ is then denoted by $\cP_{k}$. 
The space $\cM_{k} : = \ker{\upx} \cap \cP_{k}$ is called the space of spherical monogenics 
of degree $k$. Similarly,  $\cH_{k} := \ker{\Delta} \cap \cP_{k}$ is the space of spherical 
harmonics of degree $k$.

Next we define the inner product and the wedge product of two vectors $\ux$ and $\uy$
\begin{align*}
\langle \ux, \uy \rangle &:= \sum_{j=1}^{m} x_{j} y_{j} = - \frac{1}{2} (\ux \uy + \uy  \ux)\\
\ux \wedge \uy &:= \sum_{j<k} e_{j}e_{k} (x_{j} y_{k} - x_{k}y_{j}) = \frac{1}{2} (\ux \uy - \uy  \ux).
\end{align*}
For the sequel we need the square of $\ux \wedge \uy$. A short computation (see \cite{DBXu}) shows
\begin{align*}
(\ux \wedge \uy )^{2} &=  - |\ux|^{2}|\uy|^{2} +\langle \ux,\uy\rangle^{2} = - \sum_{j<k}  (x_{j} y_{k} - x_{k}y_{j})^{2},
\end{align*}
from which we observe that $(\ux \wedge \uy )^{2}$ is real-valued. 

We also introduce a basis $\{ \psi_{j,k,\ell}\}$ for the space $\cS(\mR^{m}) \otimes \cC l_{0,m}$, where 
$\cS(\mR^m)$ denotes the Schwartz space.  Define the functions  $\psi_{j,k, \ell}(x)$ by
\begin{align} \label{basis}
\begin{split}
\psi_{2j,k, \ell}(\ux) &:= L_{j}^{\frac{m}{2}+k-1}(|\ux|^{2}) M_{k}^{(\ell)} e^{-|\ux|^{2}/2},\\
\psi_{2j+1,k, \ell}(\ux) &:= L_{j}^{\frac{m}{2}+k}(|\ux|^{2}) \ux M_{k}^{(\ell)} e^{-|\ux|^{2}/2},
\end{split}
\end{align}
where $j,k \in \mN$, $\{M_{k}^{(\ell)} \in \cM_{k}: \ell = 1, \ldots, \dim \cM_{k}\}$ is a basis for $\cM_{k}$,
and $L_{j}^{\alpha}$ are the Laguerre polynomials. The set $\{ \psi_{j,k, \ell}\}$ forms a basis of 
$\cS(\mR^{m}) \otimes \cC l_{0,m}$, see \cite{MR926831}.

The operator exponential definition of the CFT $\cF_{\pm}$ is given by (see \cite{MR2190678}, we follow the normalization given in \cite{AIEP})
\begin{equation}
\cF_{\pm}:=e^{ \frac{i \pi m}{4}} e^{\frac{i \pi}{4}(\Delta - |\ux|^{2} \mp 2\Gamma)} = e^{ \frac{i \pi m}{4}} e^{\mp \frac{i \pi}{2}\Gamma  }e^{\frac{i \pi}{4}(\Delta - |\ux|^{2})}.
\label{CFTF}
\end{equation}
The second equality follows because $\Gamma$ commutes with $\Delta$ and $|\ux|^{2}$.

In the paper \cite{DBXu}, the question whether $\cF_{\pm}$ can be written as an integral transform
\begin{align*}
\cF_{\pm}[f](\uy) & : = (2 \pi)^{-\frac{m}{2}} \int_{\mR^{m}} K_\pm(\ux,\uy) f(\ux) \, dx,
\end{align*}
was answered positively in the case of even dimension $m$. The result is given in the following theorem.

\begin{thm}  \label{EvenExplicitCFT}
The kernel of the Clifford-Fourier transform in even dimension $m>2$ is given by
\begin{align*}
K_-(\ux,\uy) &= e^{i \frac{\pi}{2}\Gamma_{\uy}} e^{-i \langle \ux,\uy \rangle}\\
& = (-1)^{\frac{m}{2}} \left(\frac{\pi}{2} \right)^{\frac{1}{2}} \left(A_{(m-2)/2}^{*}(s,t) +B_{(m-2)/2}^{*}(s,t)+ (\ux \wedge \uy) \ C_{(m-2)/2}^{*}(s,t)\right)
\end{align*}
where $s=\langle \ux,\uy \rangle$ and $t= |\ux \wedge \uy| =   \sqrt{|\ux|^{2} |\uy|^{2}-s^2}$ and
\begin{align*}
A_{(m-2)/2}^{*}(s,t) &=\sum_{\ell=0}^{\left\lfloor \frac{m}{4}-\frac{3}{4}\right\rfloor}  s^{m/2-2-2\ell} \ \frac{1}{2^{\ell} \ell!} \ \frac{\Gamma \left( \frac{m}{2} \right) }{ \Gamma \left( \frac{m}{2}-2\ell -1 \right)} \widetilde{J}_{(m-2\ell-3)/2}(t),\\
B_{(m-2)/2}^{*}(s,t)&= - \sum_{\ell=0}^{\left\lfloor \frac{m}{4}-\frac{1}{2}\right\rfloor} s^{m/2-1-2\ell} \ \frac{1}{2^{\ell} \ell!} \ \frac{\Gamma \left( \frac{m}{2} \right)}{\Gamma \left( \frac{m}{2}-2\ell \right)} \widetilde{J}_{(m-2\ell-3)/2}(t),\\
C_{(m-2)/2}^{*}(s,t) &= - \sum_{\ell=0}^{\left\lfloor \frac{m}{4}-\frac{1}{2}\right\rfloor} s^{m/2-1-2\ell} \ \frac{1}{2^{\ell} \ell!} \ \frac{\Gamma \left( \frac{m}{2} \right)}{\Gamma \left( \frac{m}{2}-2\ell \right)}  \widetilde{J}_{(m-2\ell-1)/2}(t)
\end{align*}
with $\widetilde{J}_{\alpha}(t) = t^{-\alpha} J_{\alpha}(t)$.
\end{thm}

This theorem was obtained by the explicit computation of
\begin{align}
\label{relsOrdCFT}
\begin{split}
A^{*}_{k} &=  \left(\frac{\pi}{2} \right)^{-\frac{1}{2}} k z^{-k-1} \partial_{w}^{k-1} \left( w^{-1} \partial_{w} B_{0}^{0}\right)\\
B^{*}_{k} &= - \left(\frac{\pi}{2} \right)^{-\frac{1}{2}}  z^{-k} \partial_{w}^{k}  B_{0}^{0}\\
C^{*}_{k} &=  -\left(\frac{\pi}{2} \right)^{-\frac{1}{2}}  z^{-k-2} \partial_{w}^{k} \left( w^{-1} \partial_{w} B_{0}^{0}\right)
\end{split}
\end{align}
where $z=|\ux| |\uy|$, $w=\langle \underline{\xi}, \underline{\eta} \rangle$, $\ux = |\ux | \ \underline{\xi}$ , $\uy = |\uy| \ \underline{\eta}$ and $B_{0}^{0} =\cos{(z \sqrt{1-w^{2}})}$.


\section{Fractional version of the Clifford-Fourier transform}
\label{FracVersion}
\setcounter{equation}{0}

In this section, we show how a fractional version of the CFT can be introduced. To that end, we adapt the initial definition of the CFT
\[
\cF_{\pm} =  e^{ \frac{i \pi m}{4}} e^{\mp \frac{i \pi}{2} \Gamma}e^{\frac{i \pi}{4}(\Delta - |\ux|^{2})}
\]
to
\[
\cF_{\alpha, \b} =   e^{ \frac{i \alpha m}{2}} e^{i\b \Gamma}e^{ \frac{i \alpha}{2}(\Delta - |\ux|^{2})}
\]
where $\alpha, \b \in [-\pi,\pi]$. Notice that we immediately have
\[
\cF_{\alpha, \b} \circ \cF_{\alpha, - \b} = \cF_{\alpha, -\b} \circ \cF_{\alpha,  \b} =  \cF_{\alpha}^{2}
\]
with $\cF_{\alpha} =    e^{ \frac{i \alpha m}{2}}e^{ \frac{i \alpha}{2}(\Delta - |\ux|^{2} )}
 $ the fractional version of the ordinary Fourier transform, given explicitly by (see \cite{OZA})
\begin{align*}
\cF_{\alpha} [f](\uy)&= \left(\pi (1- e^{-2i \alpha})\right)^{-m/2} \int_{\mR^m}  e^{-i \la \ux,\uy \ra / \sin \alpha} e^{\frac{i}{2}( \cot \alpha) (|\ux|^2 + |\uy|^2)}f(\ux) dx.
\end{align*}

Our aim is to find an integral expression for $\cF_{ \alpha, \b}$:
\[
\cF_{ \alpha,\b}[f](\uy) =  \left(\pi (1- e^{-2i \alpha})\right)^{-m/2}\int_{\mathbb{R}^m} K_{\alpha,\b}(\ux,\uy) \ f(\ux) \ dx.
\]
In our computation, we need to put a few restrictions on the parameters $\alpha$ and $\beta$. We exclude for now the case where $\alpha =0$ or $\alpha =\pm \pi$. In section \ref{excep} we will explain what happens for these exceptional values.

We compute formally
\begin{align*}
\cF_{ \alpha,\b} &= e^{ \frac{i \alpha m}{2}} e^{ i \b \Gamma} e^{ \frac{i \alpha}{2}(\Delta - |\ux|^{2})}\\
&= \left(\pi (1- e^{-2i \alpha})\right)^{-m/2} e^{ i \b \Gamma_{\uy}} \int_{\mR^m}    e^{-i \la \ux,\uy \ra / \sin \alpha} e^{\frac{i}{2}( \cot \alpha) (|\ux|^2 + |\uy|^2)} (.) dx.
\end{align*}
Hence, the kernel is given by
\begin{align*}
K_{\alpha,\b}(\ux,\uy) &= e^{ i \b \Gamma_{\uy}} \left(  e^{-i \la \ux,\uy \ra / \sin \alpha} e^{\frac{i}{2}( \cot \alpha) (|\ux|^2 + |\uy|^2)} \right)\\
&= e^{\frac{i}{2}( \cot \alpha) (|\ux|^2 + |\uy|^2)}  e^{ i \b \Gamma_{\uy}} \left(  e^{-i \la \ux,\uy \ra / \sin \alpha}\right)
\end{align*}
where the last line follows because $\Gamma_{\uy}$ commutes with $|\uy|$. 

Recall now the series expansion for the ordinary Fourier kernel (\cite{MR0010746}, Section 11.5)
\begin{equation}
\label{planewave}
e^{-i \langle\ux, \uy\rangle} = 2^{\lambda} \Gamma(\lambda)\sum_{k=0}^{\infty}(k+ \lambda) (-i)^{k} (|\ux||\uy|)^{-\lambda} J_{k+ \lambda}(|\ux||\uy|) \; C_{k}^{\lambda}(\langle \underline{\xi},\underline{\eta} \rangle),
\end{equation}
where $\underline{\xi}= \ux/|\ux|$, $\underline{\eta} = \uy/|\uy|$ and $\lambda =(m-2)/2$. Here, $J_{\nu}$ is the Bessel function and $C_{k}^{\lambda}$ the Gegenbauer polynomial.

Using (\ref{planewave}), the term $e^{ i \b \Gamma_{\uy}} \left(  e^{-i \la \ux,\uy \ra / \sin \alpha}\right)
$ can be computed. Indeed, we find
\begin{align*}
&e^{ i \b \Gamma_{\uy}} \left(  e^{-i \la \ux,\uy \ra / \sin \alpha}\right)\\& = e^{ i \b \Gamma_{\uy}} 
 \Gamma(\lambda)\sum_{k=0}^{\infty}2^{\lambda}(k+ \lambda) (i \sin \alpha)^{-k} (|\ux||\uy|)^{k} \widetilde{J}_{k+ \lambda}\left( \frac{|\ux||\uy|}{\sin \alpha} \right) \; C_{k}^{\lambda}(\langle \underline{\xi},\underline{\eta} \rangle)\\
 &=\Gamma(\lambda)\sum_{k=0}^{\infty}2^{\lambda}(k+ \lambda) (i \sin \alpha)^{-k} \widetilde{J}_{k+ \lambda}\left( \frac{|\ux||\uy|}{\sin \alpha} \right) \;  e^{ i \b \Gamma_{\uy}}\left(  (|\ux||\uy|)^{k} C_{k}^{\lambda}(\langle \underline{\xi},\underline{\eta} \rangle)\right)
\end{align*}
so we have reduced the problem to calculating $e^{ i \b \Gamma_{\uy}}\left(  (|\ux||\uy|)^{k} C_{k}^{\lambda}(\langle \underline{\xi},\underline{\eta} \rangle)\right)$. This can be done in a manner analogous to Lemma 3.1 in \cite{DBXu}. The result is given in the following lemma.

\begin{lemma}\label{actionGamma}
One has
\begin{align*}
e^{i \b\Gamma_{\uy}}(|\ux||\uy|)^{k} C_{k}^{\lambda}(\langle \underline{\xi},\underline{\eta}  \rangle) = \, & \frac{1}{2} \left(e^{i\b(k+m-2)} + e^{-i \b k} \right)(|\ux||\uy|)^{k} C_{k}^{\lambda}(\langle\underline{\xi},\underline{\eta}  \rangle)\\
&  - \frac{\lambda }{2(k+\lambda)} \left(e^{i\b(k+m-2)}- e^{-i\b k}\right) (|\ux||\uy|)^{k} C_{k}^{\lambda}(\langle \underline{\xi},\underline{\eta}  \rangle)\\
& +  \frac{ \lambda  }{k + \lambda} \ux \wedge \uy (  e^{i \b(k+m-2)} -e^{-i \b k}) (|\ux||\uy|)^{k-1} C_{k-1}^{\lambda+1}(\langle\underline{\xi},\underline{\eta} \rangle).
\end{align*}
\end{lemma}

Using this lemma and the previous computation, we arrive at the following series representation for the kernel of the fractional CFT.

\begin{theorem}
\label{FracSeries}
The fractional Clifford-Fourier transform $\cF_{ \alpha, \b} =  e^{ \frac{i \alpha m}{2}} e^{ i \b \Gamma} e^{ \frac{i \alpha}{2}(\Delta - |\ux|^{2})}$ is given by the integral transform
\[
\left(\pi (1- e^{-2i \alpha})\right)^{-m/2}\int_{\mathbb{R}^m} K_{\alpha,\b}(\ux,\uy) \ f(\ux) \ dx
\]
with integral kernel 
\[
K_{\alpha,\b}(\ux,\uy) =  \left(A_{\lambda}^{\alpha,\b} + B_{\lambda}^{\alpha,\b} + \left(\ux \wedge \uy \right) C_{\lambda}^{\alpha,\b}\right)e^{\frac{i}{2}( \cot \alpha) (|\ux|^2 + |\uy|^2)} 
\]
with
\begin{align*}
A_{\lambda}^{\alpha,\b}(w,\widetilde{z}) =& \,- 2^{\lambda-1} \Gamma(\lambda+1)\sum_{k=0}^{\infty} i^{-k}(e^{i \b (k+2 \lambda)} - e^{-i \b k}) \widetilde{z}^{-\lambda} J_{k+ \lambda}(\widetilde{z}) \;  C_{k}^{\lambda}(w),\\
B_{\lambda}^{\alpha,\b}(w,\widetilde{z}) = & \,  2^{\lambda-1} \Gamma(\lambda)\sum_{k=0}^{\infty}(k+ \lambda) i^{-k}(e^{i \b (k+2 \lambda)} +e^{-i \b k})\widetilde{z}^{-\lambda} J_{k+ \lambda}(\widetilde{z}) \; C_{k}^{\lambda}(w),\\
C_{\lambda}^{\alpha,\b}(w,\widetilde{z}) = & \,   \frac{2^{\lambda} \Gamma(\lambda+1)}{\sin{\alpha}}\sum_{k=1}^{\infty} i^{-k}(e^{i \b (k+2 \lambda)} - e^{-i \b k}) \widetilde{z}^{-\lambda-1} J_{k+ \lambda}(\widetilde{z}) \;  C_{k-1}^{\lambda+1} (w),
\end{align*}
where $\widetilde{z} = (|\ux||\uy|)/ \sin{\alpha}$, $w = \langle \underline{\xi},\underline{\eta} \rangle$ and $\lambda = (m-2)/2$.
\end{theorem}

\begin{rem}
It is important to note that the kernels of these integral transforms are not symmetric, in the sense that $K_{\alpha,\b}(\ux,\uy) \neq K_{\alpha,\b}(\uy,\ux)$. Hence, we adopt the convention that we always integrate over the first variable in the kernel.
\end{rem}

The functions $A_{\lambda}^{\alpha,\b}$, $B_{\lambda}^{\alpha,\b}$ and $C_{\lambda}^{\alpha,\b}$ satisfy nice recursive relations. They are given in the following lemma.
\begin{lem} \label{RecursionsProps}
The following identities hold
\begin{align*}
A_{\lambda}^{\alpha,\b}(w,\widetilde{z}) =&  \frac{\lambda}{\lambda -1} \frac{i e^{i\b}}{\widetilde{z}} \partial_{w}A_{\lambda-1}^{\alpha,\b}(w,\widetilde{z})&\lambda \ge 2,\\
B_{\lambda}^{\alpha,\b}(w,\widetilde{z}) =&   \frac{i e^{i \b}}{\widetilde{z}} \partial_{w}B_{\lambda-1}^{\alpha,\b}(w,\widetilde{z})&\lambda \ge 1,\\
C_{\lambda}^{\alpha,\b}(w,\widetilde{z}) =& \frac{i e^{i \b}}{ \widetilde{z}} \partial_{w} C_{\lambda-1}^{\alpha,\b}(w,\widetilde{z})&\lambda \ge 1.
\end{align*}
\end{lem}

\begin{proof}
This follows immediately from Theorem \ref{FracSeries} and the fact that $\frac{d}{d w} C^{\lambda}_{k}(w) = 2 \lambda  C^{\lambda+1}_{k-1}(w)$.
\end{proof}

\section{Explicit representation of the kernel}
\setcounter{equation}{0}
\label{ExplicitCF}

We determine the explicit formula of $K_{\alpha,\b}(\ux,\uy)$ on $\RR^m$ in the case of even dimension. 

We first obtain the result in dimension 2, using the series expansion obtained in the previous section. Note that this result can also be proven using Clifford algebra techniques in a similar way as was done in \cite{DBDSC} for the case where $\alpha =\beta$. This is left as an exercise for the reader.

Subsequently, we derive the kernel for even dimensions larger than 2.

\subsection{The case $m=2$}
\label{CF2}


In this case $\lambda =0$. We need the well-known relation \cite[(4.7.8)]{Sz} 
$$
\lim_{\lambda \rightarrow 0}  \lambda^{-1} C_n^\lambda (w) = (2/n) \cos n \theta, \quad w = \cos{\theta}, \quad n \geq 1.   
$$
We then compute the three series given in Theorem \ref{FracSeries}.
The first one reduces to $A_0^{\alpha,\b}(w,\widetilde{z}) =0$. Next we obtain $B_0^{\alpha,\b}(w,\widetilde{z})$ as
\begin{align*}
 B_0^{\a,\b}(w,\widetilde{z}) & = J_0(\widetilde{z}) + 2 \sum_{k=1}^\infty i^{-k} J_{k}(\widetilde{z}) \cos{k \b} \cos{k \theta}\\
&= \frac{1}{2} \left(J_0(\widetilde{z}) + 2 \sum_{k=1}^\infty i^{-k} J_{k}(\widetilde{z}) \cos{k (\b-\theta)} \right)\\
& \quad + \frac{1}{2} \left(J_0(\widetilde{z}) + 2 \sum_{k=1}^\infty i^{-k} J_{k}(\widetilde{z}) \cos{k (\b+\theta)} \right)\\
&=\frac{1}{2} e^{ -i \widetilde{z} \cos{(\b+\theta)}} + \frac{1}{2} e^{ -i \widetilde{z} \cos{(\b-\theta)}}\\
&= \frac{1}{2} e^{-i z w \frac{\cos{\b}}{\sin{\a}}} \left( e^{-i z \sin{\theta}\frac{\sin{\b}}{\sin{\a}}} + e^{i z \sin{\theta}\frac{\sin{\b}}{\sin{\a}}} \right)\\
&= \cos{\left(z\sin{\theta}\frac{\sin{\b}}{\sin{\a}}\right)} e^{-i z w \frac{\cos{\b}}{\sin{\a}}}.
\end{align*}
In the 4th line we have used the well-known decomposition (see \cite{Er}, p. 7, formula (27))
\[
e^{ -i \widetilde{z} \cos{\theta}} = J_0(\widetilde{z}) + 2 \sum_{k=1}^\infty i^{-k} J_{k}(\widetilde{z}) \cos{k \theta}.
\]

Similarly we calculate $C_0^{\alpha,\b}(w,\widetilde{z})$ as
 \begin{align*}
 C_0^{\alpha,\b}(w,\widetilde{z}) & = \frac{2i}{ \widetilde{z}\sin{\alpha}} \sum_{k=1}^\infty i^{-k} J_{k}(\widetilde{z})  \frac{\sin{k\b} \sin{k\theta}}{\sin \theta}\\
 &=\frac{i}{2 z \sin{\theta}} \left( J_0(\widetilde{z})+ 2\sum_{k=1}^\infty i^{-k} J_{k}(\widetilde{z})  \cos{k(\b-\theta)}\right)\\
 & \quad - \frac{i}{2 z \sin{\theta}} \left( J_0(\widetilde{z})+ 2\sum_{k=1}^\infty i^{-k} J_{k}(\widetilde{z})  \cos{k(\b+\theta)}\right)\\
 &=\frac{i}{2 z \sin{\theta}}  \left( e^{ -i \widetilde{z} \cos{(\b-\theta)}} -  e^{ -i \widetilde{z} \cos{(\b+\theta)}} \right)\\
 &= \frac{\sin{(z\sin{\theta}\frac{\sin{\b}}{\sin{\a}})}}{z\sin{\theta}}e^{-i z w \frac{\cos{\b}}{\sin{\a}}} .
\end{align*}

Hence we have obtained the following
\begin{thm}
The kernel of the fractional Clifford-Fourier transform in dimension $m=2$ is given by
\begin{align*}
K_{\alpha,\b}(\ux,\uy) &= e^{\frac{i}{2}( \cot \alpha) (|\ux|^2 + |\uy|^2)}  e^{ i \b \Gamma_{\uy}} \left(  e^{-i \la \ux,\uy \ra / \sin \alpha}\right)\\
& = \left(\cos{ \left(t\frac{\sin{\b}}{\sin{\a}}\right)} + \left(\ux \wedge \uy \right) \frac{\sin{\left(t \frac{\sin{\b}}{\sin{\a}}\right)}}{t}\right)e^{-i \langle \ux,\uy \rangle\frac{\cos{\b}}{\sin{\a}}}e^{\frac{i}{2}( \cot \alpha) (|\ux|^2 + |\uy|^2)} 
\end{align*}
with $t= |\ux \wedge \uy| =   \sqrt{|\ux|^{2} |\uy|^{2}-\langle \ux, \uy \rangle^2}$.
\label{KernelDim2}
\end{thm}

\subsection{The case $m>2$}

Observe first that
\[
A_{1}^{\alpha,\b} = - i e^{i \b} (\sin{\alpha}) C_{0}^{ \alpha,\b}
\]
and
\begin{align*}
B_{0}^{ \alpha,\b} & = e^{-i z w \frac{\cos{\b}}{\sin{\a}}} \cos{ \left(z \sqrt{1-w^{2}}\frac{\sin{\b}}{\sin{\a}}\right)}\\
C_{0}^{ \alpha,\b} &= e^{-i z w \frac{\cos{\b}}{\sin{\a}}}\frac{\sin{\a}}{\sin{\b}}z^{-2}w^{-1} \partial_{w}  \cos{ \left(z \sqrt{1-w^{2}}\frac{\sin{\b}}{\sin{\a}}\right)}.
\end{align*}

Using the recursion relations obtained in Lemma \ref{RecursionsProps} we subsequently obtain
\begin{align*}
A_{k}^{\alpha,\b} &= - k (i e^{i \b})^{k} \frac{1}{\sin{\b}} \, \widetilde{z}^{-k-1} \partial_{w}^{k-1} \left(e^{-i z w \frac{\cos{\b}}{\sin{\a}}} w^{-1}\partial_{w}\cos{ \left(z \sqrt{1-w^{2}}\frac{\sin{\b}}{\sin{\a}}\right)} \right),\\
B_{k}^{\alpha,\b} &= (i e^{i \b})^{k}\widetilde{z}^{-k}  \partial_{w}^{k}\left(e^{-i z w \frac{\cos{\b}}{\sin{\a}}} \cos{ \left(z \sqrt{1-w^{2}}\frac{\sin{\b}}{\sin{\a}}\right)} \right), \\
C_{k}^{\alpha,\b} &= (i e^{i \b})^{k} \frac{1}{\sin{\b}\sin{\a}} \widetilde{z}^{-k-2} \partial_{w}^{k} \left(e^{-i z w \frac{\cos{\b}}{\sin{\a}}}w^{-1} \partial_{w}  \cos{ \left(z \sqrt{1-w^{2}}\frac{\sin{\b}}{\sin{\a}}\right)} \right).
\end{align*}
Introducing a new variable $z^{*} = \widetilde{z} \sin{\b}$, these formulae simplify to
\begin{align*}
A_{k}^{\alpha,\b} &= - k (i e^{i \b})^{k} (\sin{\b})^{k} \, (z^{*})^{-k-1} \partial_{w}^{k-1} \left(e^{-i z^{*} w \cot{\b}} w^{-1}\partial_{w}\cos{ \left(z^{*} \sqrt{1-w^{2}}\right)} \right),\\
B_{k}^{\alpha,\b} &=(i e^{i \b})^{k}(\sin{\b})^{k}(z^{*})^{-k}  \partial_{w}^{k}\left(e^{-i z^{*} w \cot{\b}} \cos{ \left(z^{*} \sqrt{1-w^{2}}\right)} \right), \\
C_{k}^{\alpha,\b} &=  (i e^{i \b})^{k} \frac{(\sin{\b})^{k+1}}{\sin{\a}}(z^{*})^{-k-2}  \partial_{w}^{k} \left(e^{-i z^{*} w \cot{\b}}w^{-1} \partial_{w}  \cos{ \left(z^{*} \sqrt{1-w^{2}}\right)} \right).
\end{align*}
Let us now compute $A_{k}^{\alpha,\b}$ explicitly. We find
\begin{align*}
A_{k}^{\alpha,\b} &= - k (i e^{i \b})^{k} (\sin{\b})^{k}  \sum_{\ell=0}^{k-1} 
\binom{k-1}{\ell} (z^{\ast})^{\ell-k+1} \partial_{w}^{k-\ell-1} \left(e^{-i z^{\ast} w \cot{\b}}\right) \\
& \qquad \times (z^{\ast})^{-\ell-2} \partial_{w}^{\ell} \left(w^{-1}\partial_{w} \cos{(z^{\ast} \sqrt{1-w^2})}\right) \\
&= \left(\frac{\pi}{2} \right)^{\frac{1}{2}} k\, i\, e^{i k \b} \sin{\b} \ (\cos{\b})^{k-1} e^{-i z^{\ast} w \cot{\b}} \sum_{\ell=0}^{k-1} 
\binom{k-1}{\ell} (i \tan{\b})^{\ell} \ C_{\ell}^{*}(s^{\ast},t^{\ast})
\end{align*}
where we have introduced the variables $s^{\ast} = z^{\ast} w = \frac{\sin{\b}}{\sin{\alpha}} \langle \ux,\uy \rangle$ and $t^{\ast} = z^{\ast} \sqrt{1-w^2}=\frac{\sin{\b}}{\sin{\alpha}} \ |\ux \wedge \uy| $ and used the definition of $C_{\ell}^{*}(s^{\ast},t^{\ast})$ in Theorem \ref{EvenExplicitCFT} and formula (\ref{relsOrdCFT}).

In a similar way, we can compute $B_{k}^{\alpha,\b}$ and $C_{k}^{\alpha,\b}$, yielding
\begin{displaymath}
B_{k}^{\alpha,\b} = - \left(\frac{\pi}{2} \right)^{\frac{1}{2}} e^{i k \b} (\cos{\b})^{k} e^{-i z^{\ast} w \cot{\b}} \sum_{\ell=0}^{k} 
\binom{k}{\ell} (i \tan{\b})^{\ell} \ B_{\ell}^{*}(s^{\ast},t^{\ast})
\end{displaymath}
and
\begin{displaymath}
C_{k}^{\alpha,\b} = - \left(\frac{\pi}{2} \right)^{\frac{1}{2}} e^{i k \b} \frac{\sin{\b}}{\sin{\alpha}} \ (\cos{\b})^{k} e^{-i z^{\ast} w \cot{\b}} \sum_{\ell=0}^{k} 
\binom{k}{\ell} (i \tan{\b})^{\ell} \ C_{\ell}^{*}(s^{\ast},t^{\ast}).
\end{displaymath}
This leads to the following theorem.
\begin{thm}\label{EvenExplicit}
The kernel of the fractional Clifford-Fourier transform in even dimension $m>2$ is given by
\begin{align*}
K_{\alpha,\b}(\ux,\uy) &= e^{\frac{i}{2}( \cot \alpha) (|\ux|^2 + |\uy|^2)}  e^{ i \b \Gamma_{\uy}} \left(  e^{-i \la \ux,\uy \ra / \sin \alpha}\right)\\
& =  \left(\frac{\pi}{2} \right)^{\frac{1}{2}} e^{i \b(m-2)/2} (\cos{\b})^{(m-2)/2} e^{-i s^{\ast} \cot{\b}} e^{\frac{i}{2}( \cot \alpha) (|\ux|^2 + |\uy|^2)} \\
&\qquad \times \left(A_{(m-2)/2}^{\alpha,\b, *}(s^{\ast},t^{\ast}) +B_{(m-2)/2}^{\alpha,\b, *}(s^{\ast},t^{\ast})+ (\ux \wedge \uy) \ C_{(m-2)/2}^{\alpha,\b, *}(s^{\ast},t^{\ast})\right) 
\end{align*}
where $s^{\ast}= \frac{\sin{\b}}{\sin{\alpha}} \langle \ux,\uy \rangle$ and $t^{\ast}= \frac{\sin{\b}}{\sin{\alpha}} |\ux \wedge \uy|$ and
\begin{align*}
A_{(m-2)/2}^{\alpha,\b,*}(s^{\ast},t^{\ast}) &= i \ \left( \frac{m-2}{2}\right) \tan{\b} \sum_{\ell=0}^{\frac{m}{2}-2} 
\binom{\frac{m}{2}-2}{\ell} (i \tan{\b})^{\ell} \ C_{\ell}^{*}(s^{\ast},t^{\ast}),\\
B_{(m-2)/2}^{\alpha,\b,*}(s^{\ast},t^{\ast})&= - \sum_{\ell=0}^{\frac{m}{2}-1} 
\binom{\frac{m}{2}-1}{\ell} (i \tan{\b})^{\ell} \ B_{\ell}^{*}(s^{\ast},t^{\ast}),\\
C_{(m-2)/2}^{\alpha,\b,*}(s^{\ast},t^{\ast}) &= -  \frac{\sin{\b}}{\sin{\alpha}} \sum_{\ell=0}^{\frac{m}{2}-1} 
\binom{\frac{m}{2}-1}{\ell} (i \tan{\b})^{\ell} \ C_{\ell}^{*}(s^{\ast},t^{\ast})
\end{align*}
with $B_{\ell}^{*}(s^{\ast},t^{\ast})$ and $C_{\ell}^{*}(s^{\ast},t^{\ast})$ defined in Theorem \ref{EvenExplicitCFT} and with $\widetilde{J}_{\alpha}(t^{\ast}) = (t^{\ast})^{-\alpha} J_{\alpha}(t^{\ast})$.
\end{thm}

Taking into account the expression of $C_{\ell}^{*}(s^{\ast},t^{\ast})$ and executing the substitution $\ell = p+j$, we obtain consecutively
\begin{align*}
& \sum_{\ell=0}^{k-1} 
\binom{k-1}{\ell} (i \tan{\b})^{\ell} C_{\ell}^{*}(s^{\ast},t^{\ast})\\
& =- \sum_{\ell=0}^{k-1} \sum_{j=0}^{\left\lfloor \frac{\ell}{2}\right\rfloor} \binom{k-1}{\ell} (i \tan{\b})^{\ell} (s^{\ast})^{\ell-2j} \frac{\ell!}{2^{j}j! (\ell-2j)!} \widetilde{J}_{\ell-j + \frac{1}{2}}(t^{\ast})\\
&= - \sum_{p=0}^{k-1} \left(\sum_{j=0}^{\min{(p,k-1-p)}} \frac{(k-1)!}{(k-1-p-j)! j! (p-j)!} \left(\frac{i \tan{\b}}{2s^{\ast}}\right)^{j} \right)(i s^{\ast} \tan{\b})^{p}  \widetilde{J}_{p+ \frac{1}{2}}(t^{\ast}).
\end{align*}
Similarly, we find that
\begin{align*}
& \sum_{\ell=0}^{k} 
\binom{k}{\ell} (i \tan{\b})^{\ell} B_{\ell}^{*}(s^{\ast},t^{\ast})\\
&= - \sum_{p=0}^{k} \left(\sum_{j=0}^{\min{(p,k-p)}} \frac{k!}{(k-p-j)! j! (p-j)!} \left(\frac{i \tan{\b}}{2s^{\ast}}\right)^{j} \right)(i s^{\ast} \tan{\b})^{p}  \widetilde{J}_{p- \frac{1}{2}}(t^{\ast}).
\end{align*}
Hence, we obtain the following explicit formulae for $A_{(m-2)/2}^{\alpha,\b,*}$, $B_{(m-2)/2}^{\alpha,\b,*}$ and $C_{(m-2)/2}^{\alpha,\b,*}$ in terms of a finite sum of Bessel functions:
\begin{align*}
A_{(m-2)/2}^{\alpha,\b,*} (s^{\ast},t^{\ast})&= - i \left( \frac{m-2}{2} \right) \tan{\b} \sum_{p=0}^{m/2-2} \left(\sum_{j=0}^{\min{(p,m/2-2-p)}} \frac{(m/2-2)!}{(m/2-2-p-j)! j! (p-j)!} \left(\frac{i \tan{\b}}{2s^{\ast}}\right)^{j} \right)\\
& \times (i s^{\ast} \tan{\b})^{p}  \widetilde{J}_{p+ \frac{1}{2}}(t^{\ast})\\
B_{(m-2)/2}^{\alpha,\b,*} (s^{\ast},t^{\ast})&= \sum_{p=0}^{m/2-1} \left(\sum_{j=0}^{\min{(p,m/2-1-p)}} \frac{(m/2-1)!}{(m/2-1-p-j)! j! (p-j)!} \left(\frac{i \tan{\b}}{2s^{\ast}}\right)^{j} \right)\\
& \times (i s^{\ast} \tan{\b})^{p}  \widetilde{J}_{p- \frac{1}{2}}(t^{\ast})\\
C_{(m-2)/2}^{\alpha,\b,*}(s^{\ast},t^{\ast})  &=  \frac{\sin{\b}}{\sin{\alpha}}  \sum_{p=0}^{m/2-1} \left(\sum_{j=0}^{\min{(p,m/2-1-p)}} \frac{(m/2-1)!}{(m/2-1-p-j)! j! (p-j)!} \left(\frac{i \tan{\b}}{2s^{\ast}}\right)^{j} \right)\\
& \times (i s^{\ast} \tan{\b})^{p}  \widetilde{J}_{p+ \frac{1}{2}}(t^{\ast}).
\end{align*}

\section{Fractional Clifford-Fourier system}
\setcounter{equation}{0}
\label{secFurtherprops}

In this section we derive the system of PDEs that is satisfied by the kernel of the fractional CFT. 

\begin{prop} \label{DiffKernel}
For all $m$, the kernel $K_{\alpha,\b}(\ux,\uy)$ satisfies the properties
\begin{align*}
(i \sin{\alpha} \ \partial_{\uy} + \cos{\alpha} \ \uy) K_{\alpha,\b}(\ux,\uy) &= e^{i\b(m-1)} K_{\alpha,-\b}(\ux,\uy) \ \ux,\\
\uy K_{\alpha,\b}(\ux,\uy)&= e^{i\b(m-1)} K_{\alpha,-\b}(\ux,\uy)  (i \sin{\alpha} \ \partial_{\ux} + \cos{\alpha} \ \ux).
\end{align*}
In the last formula, the Dirac operator $\upx$ is acting from the right on $K_{\alpha,-\b}$.
\end{prop}
\begin{proof}
Taking into account that
\begin{displaymath}
\partial_{\uy} \Gamma_{\uy}^k = (m-1-\Gamma_{\uy})^k \partial_{\uy}
\end{displaymath}
and thus also that
\begin{displaymath}
\partial_{\uy} e^{i\beta \Gamma_{\uy}} = e^{i \beta (m-1-\Gamma_{\uy})} \partial_{\uy},
\end{displaymath}
we subsequently calculate
\begin{align*}
\partial_{\uy} K_{\alpha,\beta}(\ux,\uy) &= \partial_{\uy} e^{\frac{i}{2}( \cot \alpha) (|\ux|^2 + |\uy|^2)}  e^{ i \b \Gamma_{\uy}}   e^{-i \la \ux,\uy \ra / \sin \alpha}\\
&= \partial_{\uy} \left( e^{\frac{i}{2}( \cot \alpha) (|\ux|^2 + |\uy|^2)} \right)  e^{ i \b \Gamma_{\uy}}   e^{-i \la \ux,\uy \ra / \sin \alpha}\\
& + e^{\frac{i}{2}( \cot \alpha) (|\ux|^2 + |\uy|^2)}  e^{i \beta (m-1-\Gamma_{\uy})} \partial_{\uy} e^{-i \la \ux,\uy \ra / \sin \alpha}\\
&= i  \cot \alpha \ \uy \ K_{\alpha,\beta}(\ux,\uy)  + e^{i \beta (m-1)} \ K_{\alpha,-\beta}(\ux,\uy) \left( - \frac{i \ux}{\sin{\alpha}} \right)
\end{align*}
or
\begin{displaymath}
(i \sin{\alpha} \ \partial_{\uy} + \cos{\alpha} \ \uy) K_{\alpha,\beta}(\ux,\uy) = e^{i\beta (m-1)} \ K_{\alpha,-\beta}(\ux,\uy) \ \ux.
\end{displaymath}
The expression for $\uy K_{\alpha,\beta}(\ux,\uy)$ is proven in a similar way using
\begin{displaymath}
\uy \Gamma_{\uy}^k = (m-1-\Gamma_{\uy})^k \uy .
\end{displaymath}
\end{proof}

\begin{remark}
When $m$ is even, we can also obtain Proposition \ref{DiffKernel} using the explicit expression for the kernel as given in Theorem \ref{EvenExplicit}. However, the resulting computation is much more tedious.
\end{remark}

Next, we obtain the following corollary.

\begin{corollary}
For all $m$, the kernel $K_{\alpha,\b}(\ux,\uy)$ satisfies
\begin{align*}
\left( \upy + \uy \right) K_{\alpha,\b}(\ux,\uy) &= - e^{-i \alpha + i \beta(m-1)} K_{\alpha,-\b}(\ux,\uy) \left( \upx - \ux \right) \\
\left( \upy - \uy \right) K_{\alpha,\b}(\ux,\uy) &= - e^{i \alpha + i \beta(m-1)} K_{\alpha,-\b}(\ux,\uy) \left( \upx + \ux \right) 
\end{align*}
and
\[
H_{x} K_{\alpha,\b}(\ux,\uy) = H_{y} K_{\alpha,\b}(\ux,\uy)
\]
with $H_{x} = \Delta_{x} - |\ux|^{2}$.
\end{corollary}

\begin{proof}
Immediately, using Proposition \ref{DiffKernel} and the fact that
\[
2 H_{x} = - \left\{ \upx+ \ux, \upx - \ux \right\}.
\]
\end{proof}

Putting
\begin{equation}\label{exprkernel}
K_{\alpha,\beta}(\ux,\uy) = \widehat{K}_{\alpha,\beta}(\ux,\uy) \  e^{\frac{i}{2}( \cot \alpha) (|\ux|^2 + |\uy|^2)},
\end{equation} 
let us determine the system of partial differential equations satisfied by $\widehat{K}_{\alpha,\beta}$.
\begin{prop}
\label{simpleSystem}
 For all $m$, $\widehat{K}_{\alpha,\beta}(\ux,\uy)$ satisfies
\begin{align*}
(i \sin{\alpha} \ \partial_{\uy}) \lbrack \widehat{K}_{\alpha,\b}(\ux,\uy) \rbrack &= e^{i\b(m-1)} \widehat{K}_{\alpha,-\b}(\ux,\uy) \ \ux,\\
\uy \widehat{K}_{\alpha,\b}(\ux,\uy)&= e^{i\b(m-1)}  \lbrack \widehat{K}_{\alpha,-\b}(\ux,\uy) \rbrack  (i \sin{\alpha} \ \partial_{\ux}).
\end{align*}
\end{prop}
\begin{proof} These equations are obtained by plugging the form (\ref{exprkernel}) into Proposition \ref{DiffKernel}.
\end{proof}

The interested reader may at this point wonder whether $\widehat{K}_{\alpha,\beta}(\ux,\uy)$ as given in Theorem \ref{EvenExplicit} is the only (unique) solution to the system of PDEs given in Proposition \ref{simpleSystem}. It turns out that this is not the case. Indeed, when $\a=\b = \pi/2$ we have investigated this phenomenon in \cite{DBNS}, yielding an entire class of solutions to this type of PDE. A similar analysis can be performed for arbitrary $\a$ and $\b$.

\section{Properties of the fractional CFT}
\setcounter{equation}{0}
\label{PropertiesSection}

From the explicit expression for the kernel of the fractional CFT given in Theorem \ref{EvenExplicit}, we can derive some properties of the kernel.

Let us start by giving a bound for the kernel.
\begin{lem}\label{bound}
Let $m$ be even. For $\ux,\uy \in \mathbb{R}^m$, there exists a constant $c$ such that
\begin{align*}
|A_{(m-2)/2}^{\alpha,\b,*} (s^{\ast},t^{\ast}) + B_{(m-2)/2}^{\alpha,\b,*} (s^{\ast},t^{\ast})| & \leq c \ (1+|\ux|)^{(m-2)/2} \ (1+|\uy|)^{(m-2)/2}\\
|(x_j y_k - x_k y_j) C_{(m-2)/2}^{\alpha,\b,*} (s^{\ast},t^{\ast})| & \leq c \ (1+|\ux|)^{(m-2)/2} \ (1+|\uy|)^{(m-2)/2}, \qquad j \not= k.
\end{align*}
\end{lem}
\begin{proof}
The proof is similar of the one of Lemma 5.2 and Theorem 5.3 in \cite{DBXu}.
\end{proof}
\noindent Note that the kernel is bounded if $m=2$.

Recall that the kernel $K_{\alpha,\b}(\ux,\uy)$ is a Clifford algebra valued function. It can be decomposed as
\begin{equation}
K_{\alpha,\b}(\ux,\uy) = K_{0}(\ux,\uy) + \sum_{i<j} e_{i} e_{j} K_{ij}(\ux,\uy)
\label{DecompKernel}
\end{equation}
with $ K_{0}(\ux,\uy)$ and $K_{ij}(\ux,\uy)$ scalar functions. Now, using Lemma \ref{bound}, we immediately have the following bounds.

\begin{thm} \label{kernelBound}
Let $m$ be even. For $\ux,\uy \in \RR^m$, one has
\begin{align*}
| K_{0}(\ux,\uy)| &\leq  c (1+|\ux|)^{(m-2)/2}(1+|\uy|)^{(m-2)/2},  \\
| K_{ij}(\ux,\uy)| &\leq c (1+|\ux|)^{(m-2)/2}(1+|\uy|)^{(m-2)/2}, \qquad
 i \neq j. 
\end{align*}
\end{thm}

The bound of the kernel function defines the domain of the fractional CFT.
\begin{thm}
Let $m$ be an even integer. The fractional Clifford-Fourier transform is well-defined on $B(\mathbb{R}^m) \otimes \cC l_{0,m}$ with
\begin{displaymath}
B(\mathbb{R}^m) = \left\lbrace f \in L^1(\mathbb{R}^m) \ : \ \int_{\mathbb{R}^m} (1+|\uy|)^{(m-2)/2} \ |f(\uy)| \ dy < \infty \right\rbrace.
\end{displaymath}
\end{thm}
\begin{proof}
This follows immediately from Theorem \ref{kernelBound}.
\end{proof}

Subsequently, we derive the calculus properties of the fractional CFT. 

\begin{lem} \label{DiffTransform}
Let $m$ be even and $f \in \cS(\mR^{m}) \otimes \cC l_{0, m}$. Then  
\begin{align*}
\cF_{\alpha, -\b} \left[(\ux\cos{\a}-i \sin{\a} \upx) \, f \right] & = e^{-i \b(m-1)} \uy \cF_{\alpha, \b} [f], \\
\cF_{\alpha, -\b} \left[(\cos{\a} \upx - i \ux \sin{\a}) f \right] & = e^{-i \b(m-1)} \upy \cF_{\alpha, \b}[f].
\end{align*} 
\end{lem}

\begin{proof}
Because $m$ is even, the kernel $K_{\a,\b}$ has a polynomial bound according to Theorem \ref{kernelBound} and we can apply integration by parts. Using Proposition \ref{DiffKernel} then yields the desired results. 
\end{proof}

We now arrive at the main theorem of this section.

\begin{thm} \label{CFschwartz} 
Let $m$ be even. Then $\cF_{\alpha, \b}$ is a continuous operator on $\cS(\RR^m) \otimes \cC l_{0,m}$.
\end{thm}

\begin{proof}
Using formula (\ref{DecompKernel}) we can rewrite $\cF_{\alpha, \b}$ as 
\[
\cF_{\alpha, \b} = \cF_{0} + \sum_{i<j} e_{i} e_{j} \cF_{ij}
\]
with
\begin{align*}
\cF_{0}[f](\uy) &= \left(\pi (1- e^{-2i \alpha})\right)^{-m/2} \int_{\mR^{m}} K_{0}(\ux,\uy) f(\ux)dx\\
\cF_{ij}[f](\uy) &= \left(\pi (1- e^{-2i \alpha})\right)^{-m/2} \int_{\mR^{m}} K_{ij}(\ux,\uy) f(\ux) dx
\end{align*}
scalar integral transforms.

Now we observe that $(\ux\cos{\a}-i \sin{\a} \upx)^{2}$, resp. $(\cos{\a} \upx - i \ux \sin{\a})^{2}$ are scalar operators. Indeed, we can e.g. compute
\begin{align*}
(\ux\cos{\a}-i \sin{\a} \upx)^{2} & = -(\cos{\a})^{2}|\ux|^{2} + (\sin{\a})^{2} \Delta - i \sin{\a} \cos{\a} \{\ux, \upx\}\\
&=-(\cos{\a})^{2}|\ux|^{2} + (\sin{\a})^{2} \Delta + i \sin{\a} \cos{\a}(2 \mE+m),
\end{align*}
with $\mE = \sum_{j=1}^{m} x_j \partial_{x_j}$ the Euler operator. Subsequently applying Lemma \ref{DiffTransform} two times, we obtain
\begin{align}
\label{propF0}
\begin{split}
\cF_{0}[(\ux\cos{\a}-i \sin{\a} \upx)^{2} f](\uy) &= -|\uy|^{2}\cF_{0}[f](\uy)\\
\cF_{0}[(\cos{\a} \upx - i \ux \sin{\a})^{2}f](\uy) &= -\Delta_{y} \cF_{0}[f](\uy).
\end{split}
\end{align}
The same results hold for $\cF_{ij}$.

It clearly suffices to prove that $\cF_{0}$ and $\cF_{ij}$ are continuous maps on $\cS(\RR^m)$. We give the proof for $\cF_{0}$, the other cases being similar.

Recall that the Schwartz class $\cS(\RR^m)$ is endowed with the topology defined by the family of semi-norms 
$$
    \rho_{\g,\d} (f) : = \sup_{x\in \RR^m} |\ux^\g \partial^\d f(\ux)|, \qquad \g ,\d \in \NN_0^m,  
$$
and $f \in \cS(\RR^m)$ if $\rho_{\g,\d}(f) < \infty$ for all $\g,\d$. An equivalent characterization is given by $\rho_{\g,n}^* (f) < \infty$ for 
$$
   \rho_{\g,n}^* (f) : =  \sup_{x\in \RR^m} |\ux^\g \Delta^n f(\ux)|, \qquad \g \in \NN_0^m, \quad n \in \NN_0,
$$
see \cite{DBXu}, proof of Theorem 6.3.

Now let $\g \in \NN_0^m$ and $n \in \NN_0$. If $|\uy| \le 1$, then by (\ref{propF0}) and Theorem \ref{kernelBound}, 
\begin{align*}
   |\uy^\g \Delta_y^n \left(\cF_{0} f \right)(\uy)|  & =  |\uy^\g |  \cdot | \cF_{0}[ (\cos{\a} \upx - i \ux \sin{\a})^{2n}f](\uy)| \\ 
   &\le c (1+|\uy|)^{(m-2)/2} \int_{\RR^m}  (1+|\ux|)^{(m-2)/2} | (\cos{\a} \upx - i \ux \sin{\a})^{2n} f(\ux)| dx \\
   & \le c_{1} \sup_{x\in \RR^m}|(1+|\ux|)^{3m/2}(\cos{\a} \upx - i \ux \sin{\a})^{2n} f(\ux)|
\end{align*}
as $f$ is a Schwartz class function. For $|\uy| \ge 1$ we find similarly
\begin{align*}
   |\uy^\g \Delta_y^n\left( \cF_{0} f\right)(\uy)|  & =  |\uy^\g |  \cdot |\cF_{0}[ (\cos{\a} \upx - i \ux \sin{\a})^{2n} f](\uy)| \\ 
   &= |\uy^\g|\cdot |\uy|^{-2 \s} |\cF_{0}[(\ux\cos{\a}-i \sin{\a} \upx)^{2\s} (\cos{\a} \upx - i \ux \sin{\a})^{2n} f](\uy) | \\
   & \le c |\uy|^{|\g| - 2\s}  (1+|\uy|)^{(m-2)/2} \,\\
   & \quad \times  \int_{\RR^m}  (1+|\ux|)^{(m-2)/2} \left|(\ux\cos{\a}-i \sin{\a} \upx)^{2\s} (\cos{\a} \upx - i \ux \sin{\a})^{2n} f(\ux) \right|dx  \\
   &  \le c_{2} \sup_{x\in \RR^m}|(1+|\ux|)^{3m/2} (\ux\cos{\a}-i \sin{\a} \upx)^{2\s} (\cos{\a} \upx - i \ux \sin{\a})^{2n} f(\ux)|
\end{align*}
if $2 \s \ge |\g| + (m-2)/2$. 

As now both terms
\begin{align*}
& \sup_{x\in \RR^m}|(1+|\ux|)^{3m/2}(\cos{\a} \upx - i \ux \sin{\a})^{2n} f(\ux)|\\
&\sup_{x\in \RR^m}|(1+|\ux|)^{3m/2} (\ux\cos{\a}-i \sin{\a} \upx)^{2\s} (\cos{\a} \upx - i \ux \sin{\a})^{2n} f(\ux)|
\end{align*}
can be expanded so that they are bounded from above by a finite sum of seminorms $ \rho_{\g,\d} (f)$, it is easy to obtain an expression of the form
\begin{align*}
&\sup_{y\in \RR^m}|\uy^\g \Delta_y^n \left( \cF_{0} f\right)(\uy)|\\ 
&\le \max \left\{c_{1} \sup_{x\in \RR^m}|(1+|\ux|)^{3m/2}(\cos{\a} \upx - i \ux \sin{\a})^{2n} f(\ux)|, \right.\\
&\quad \left.   c_{2} \sup_{x\in \RR^m}|(1+|\ux|)^{3m/2} (\ux\cos{\a}-i \sin{\a} \upx)^{2\s} (\cos{\a} \upx - i \ux \sin{\a})^{2n} f(\ux)|
\right\}.\\
&\leq \sum_{\mbox{finite}} \rho_{\g,\d} (f).
\end{align*}
This proves the continuity of $\cF_{0}$. Similar considerations give the continuity of $\cF_{ij}$, thus completing the proof of the theorem.
\end{proof}

\section{Eigenvalues of fractional CFT}
\setcounter{equation}{0}
\label{EigenvaluesSection}

In this section we will calculate the action of the fractional CFT on the basis $\lbrace \psi_{j,k,\ell} \rbrace$ defined by (\ref{basis}). This will allow us to prove the inversion theorem on Schwartz space for even dimension (see Theorem \ref{InvTheoremS}). We also discuss what happens for exceptional values of the fractional parameters $\alpha$ and $\beta$.
\subsection{Operator exponential approach}
The fractional CFT can be written as the operator exponential
\begin{displaymath}
\cF_{ \alpha, \b} = e^{i(-\alpha {\mathcal H}+\b \Gamma)}
\end{displaymath}
with ${\mathcal H} = \frac{1}{2} (- \Delta + |\ux|^2-m)$.
Combining (see \cite{AIEP}, p. 114)
\begin{displaymath}
{\mathcal H} \lbrack \psi_{j,k,\ell} \rbrack = (j+k) \ \psi_{j,k,\ell}
\end{displaymath}
with 
\begin{displaymath}
\Gamma  \lbrack \psi_{2j,k,\ell} \rbrack = -k  \psi_{2j,k,\ell} \quad ; \quad  \Gamma  \lbrack \psi_{2j+1,k,\ell} \rbrack = (k+m-1)  \psi_{2j+1,k,\ell}, \end{displaymath}
we obtain
\begin{align*}
(-\alpha {\mathcal H}+\b \Gamma) \lbrack \psi_{2j,k,\ell} \rbrack   &= \left( - \alpha (2j+k)-\b k \right) \ \psi_{2j,k,\ell}\\
(-\alpha {\mathcal H}+\b \Gamma) \lbrack \psi_{2j+1,k,\ell} \rbrack  &= \left( - \alpha (2j+1+k)+\b (k+m-1) \right) \ \psi_{2j+1,k,\ell}.
\end{align*} 
Hence, we find consecutively
\begin{align*}
e^{i(-\alpha {\mathcal H}+\b \Gamma)} \lbrack \psi_{2j,k,\ell} \rbrack &= \sum_{n=0}^{\infty} \frac{i^n}{n!}  (-\alpha {\mathcal H}+\b \Gamma)^n \lbrack \psi_{2j,k,\ell} \rbrack\\
&= \sum_{n=0}^{\infty} \frac{i^n}{n!} \left( - \alpha (2j+k)-\b k \right)^n \ \psi_{2j,k,\ell}\\
&= e^{-i\alpha (2j+k)} \ e^{-i\b k} \ \psi_{2j,k,\ell}
\end{align*}
and similarly
\begin{displaymath}
e^{i(-\alpha {\mathcal H}+\b \Gamma)} \lbrack \psi_{2j+1,k,\ell} \rbrack = e^{-i\alpha (2j+1+k)} \ e^{i\b (k+m-1)} \ \psi_{2j+1,k,\ell}.
\end{displaymath}
\subsection{Series approach}
In this subsection we consider a general kernel of the following form
\begin{equation}\label{structure kernel}
K(\ux,\uy) = \left( A(w, \widetilde{z}) + (\ux \wedge \uy) \  B(w, \widetilde{z}) \right) \ e^{\frac{i}{2}( \cot \alpha) (|\ux|^2 + |\uy|^2)} 
\end{equation}
with
\begin{align*}
A(w,\widetilde{z}) &= \sum_{k=0}^{+\infty} \alpha_{k} \ (\widetilde{z})^{-\l}J_{k+\l}(\widetilde{z}) C^{\l}_{k}(w)\\
B(w,\widetilde{z}) &=  \sum_{k=1}^{+\infty} \beta_{k} \ (\widetilde{z})^{-\l-1} J_{k+\l}(\widetilde{z}) C^{\l+1}_{k-1}(w)
\end{align*}
and $\alpha_{k}, \beta_{k} \in \mC$, $\widetilde{z} = (|\ux||\uy|)/ \sin{\alpha}$, $w=\langle \underline{\xi},\underline{\eta} \rangle$ ($\ux = |\ux| \underline{\xi}$, $\uy = |\uy| \underline{\eta}$, $\underline{\xi},\underline{\eta} \in S^{m-1}$), $\lambda=(m-2)/2$.

We define the integral transform
\begin{align*}
\cF\lbrack f \rbrack (\uy) &= \frac{1}{(\pi(1-e^{-2i\alpha}))^{m/2}} \int_{\mR^{m}} K(\ux,\uy) \ f(\ux) \ dx.
\end{align*}
Now we calculate the action of this transform on the basis (\ref{basis}) of $\cS(\mR^{m}) \otimes \cC l_{0,m}$. We start with the following auxiliary result expressing the radial behavior of the integral transform.

\begin{proposition}
\label{radbehavior1}
Let $M_{k} \in \cM_{k}$ be a spherical monogenic of degree $k$. Let $f(\ux)= f_0(|\ux|)$ be a real-valued radial function in 
$\cS(\RR^m)$. Further, put $\underline{\xi}= \ux/|\ux|$, $\underline{\eta} = \uy/|\uy|$ and $r = |\ux|$. Then one has
\begin{eqnarray*} 
\cF \left\lbrack f(r)M_{k}(\ux) \right\rbrack (\uy) &=& c_m \left( \frac{\l}{\l+k} \alpha_{k} - \sin{\alpha} \ \frac{k}{2(k+ \l)} \beta_k \right)  e^{\frac{i}{2} (\cot \alpha) |\uy|^2} M_{k}(\underline{\eta})\\
&&\times \int_{0}^{+\infty} r^{m+k-1}f_0(r)  \ (\widetilde{z})^{-\l} J_{k + \l}(\widetilde{z})  \ e^{\frac{i}{2} (\cot \alpha) r^2} dr
\end{eqnarray*}
and
\begin{eqnarray*}
\cF \left\lbrack f(r) \ux M_{k}(\ux) \right\rbrack (\uy) &=& c_m \left( \frac{\l}{\l+k+1} \alpha_{k+1} + \sin{\alpha} \frac{k+1+2\l}{2(k+1+ \l)} \beta_{k+1} \right) e^{\frac{i}{2} (\cot \alpha) |\uy|^2}  \\
&&\times \underline{\eta} \  M_{k}(\underline{\eta}) \ \int_{0}^{+\infty} r^{m+k}f_0(r)   \ (\widetilde{z})^{-\l} J_{k +1+ \l}(\widetilde{z}) \ e^{\frac{i}{2} (\cot \alpha) r^2} dr
\end{eqnarray*}
with $\widetilde{z}= \frac{r |\uy|}{\sin{\alpha}}$, $\l = (m-2)/2$ and 
\[
c_m =  \frac{2}{\Gamma \left( \frac{m}{2} \right) \ (1-e^{-2i\alpha})^{m/2}}.
\]
\label{RadialBeh}
\end{proposition}

\begin{proof}
The proof goes along similar lines as the proof of Theorem 6.4 in \cite{DBXu}.
\end{proof}

We then have the following theorem.

\begin{theorem}
\label{eigenvalues}
One has, putting $\beta_{0}=0$,
\begin{align*}
\begin{split}
\cF \lbrack \psi_{2j,k,\ell}\rbrack (\uy) & = \frac{2^{-\l}}{\Gamma(\l+1)} \ \left(\frac{\l}{\l+k} \alpha_{k} - \sin{\alpha} \ \frac{k}{2(\l+k)} \beta_{k} \right) i^k e^{-i \alpha (k+2j)} \ \psi_{2j,k,\ell}(\uy)\\
\cF \lbrack \psi_{2j+1,k,\ell}\rbrack (\uy) & = \frac{2^{-\l}}{\Gamma(\l+1)} \ \left(\frac{\l}{\l+k+1} \alpha_{k+1} + \sin{\alpha} \ \frac{k+1+2\l}{2(\l+k+1)} \beta_{k+1} \right) i^{k+1} \ e^{-i\alpha(k+2j+1)} \ \psi_{2j+1,k,\ell}(\uy).
\end{split}
\end{align*}
\end{theorem}
\begin{proof}
This follows from the explicit expression (\ref{basis}) of the basis and the identity (see \cite[p. 847, formula 7.421, number 4 with $\alpha=1$]{Grad}):
\[
\int_{0}^{+\infty} x^{\nu+1} e^{-\b x^2} L_{n}^{\nu}(x^{2}) J_{\nu}(xy) dx = 2^{-\nu-1} \b^{-\nu-n-1} (\b-1)^n y^{\nu} e^{-\frac{y^2}{4 \b}} L_n^{\nu} \left\lbrack \frac{y^2}{4 \b (1-\b)} \right\rbrack. 
\]
\end{proof}

The fractional Clifford-Fourier kernel $K_{\alpha,\b}$ has the same structure as the kernel in (\ref{structure kernel}), because of Theorem \ref{FracSeries}. The action of the fractional CFT $\cF_{\alpha,\b}$ on the basis $\lbrace \psi_{j,k,\ell} \rbrace$ can hence be determined by substituting the corresponding coefficients $\alpha_k$ and $\beta_k$ in Theorem \ref{eigenvalues}.
This yields the following result.
\begin{theorem}
\label{InvTheoremS}
For the basis $\{\psi_{j,k,\ell}\}$ of  $\cS(\mR^{m}) \otimes \cC l_{0,m}$ , one has
\begin{align*}
\cF_{\alpha,\b}\lbrack \psi_{2j,k,\ell} \rbrack &= e^{-i\alpha(2j+k)} e^{-i\b k} \ \psi_{2j,k,\ell},\\
\cF_{\alpha,\b} \lbrack \psi_{2j+1,k,\ell} \rbrack &= e^{-i\alpha(2j+1+k)} e^{i\b (k+m-1)} \ \psi_{2j+1,k,\ell}.
\end{align*}
In particular, the action of $\CF_{\alpha,\b}$ coincides with the operator $e^{i(-\alpha {\mathcal H}+\b \Gamma)}$ when restricted to the basis $\{ \psi_{j,k,\ell}\}$ and
\begin{equation}\label{inversion}
\cF_{\alpha,\b} \cF_{-\alpha,-\b}= Id 
\end{equation}
on the basis $\{\psi_{j,k,\ell}\}$. Moreover, when $m$ is even, \eqref{inversion} holds for all $f \in \cS(\RR^m) \otimes \cC l_{0,m}$.
\end{theorem}

\begin{proof}
We only need to prove the last statement, for $m$ being even. This follows from Theorem \ref{CFschwartz} and the fact that 
$\{\psi_{j,k,\ell}\}$ is a dense subset of  $\cS(\mR^{m}) \otimes \cC l_{0,m}$. 
\end{proof}

\subsection{Exceptional parameters and the fractional CFT}
\label{excep}

Using Theorem \ref{InvTheoremS} we are able to explain what happens for the exceptional parameter values $\alpha =0$ and $\alpha = \pm \pi$.
First we write the fractional CFT as the composition of two operators:
\[
\cF_{\alpha,\b} =\cF_{0,\b}  \cF_{\alpha,0}
\]
with $\cF_{0,\b} =   e^{i\b \Gamma}$ and $\cF_{\alpha,0} =   e^{ \frac{i \alpha m}{2}} e^{ \frac{i \alpha}{2}(\Delta - |\ux|^{2})}$.

When $\alpha =0$, we observe that $\cF_{\alpha,0}$ becomes the identity operator. Using the eigenvalues computed in Theorem \ref{InvTheoremS}, we moreover find that 
\[
\cF_{\pm \pi,0} [f](\uy) = f(-\uy).
\]
Note that in both cases, the resulting operators can no longer be written as integral operators with kernel given by Theorem \ref{EvenExplicit}.

We still need to consider the operator $\cF_{0,\b}=e^{i\b \Gamma}$ for these cases. This is a purely angular operator, which can be written as a singular integral operator acting on the sphere. Its detailed study will be presented elsewhere.


\end{document}